\documentclass[leqno,11pt]{amsart}

\usepackage[top=1.5in,bottom=1.3in,left=1.3in,right=1.3in,marginparwidth=1.5cm]{geometry}

\usepackage{times}
\usepackage{tensor}
\usepackage[bbgreekl]{mathbbol}
\usepackage[all]{xy}
\usepackage{tikz}
\usepackage{tikz-cd}
\usepackage{subcaption}

\usetikzlibrary{arrows}

\usepackage{amsmath, amssymb, amsfonts, latexsym, mdwlist, amsthm}
\usepackage{mathrsfs}
\usepackage{graphicx}
\usepackage{wrapfig}
\usepackage{longtable}
\usepackage{enumerate}

\makeindex

\usepackage[colorlinks=true,citecolor=blue, urlcolor=blue, linkcolor=blue, pagebackref]{hyperref}

\def\C{\ensuremath{\mathbb{C}}}
\def\P{\ensuremath{\mathbb{P}}}
\def\Q{\ensuremath{\mathbb{Q}}}
\def\R{\ensuremath{\mathbb{R}}}
\def\Z{\ensuremath{\mathbb{Z}}}

\def\tH{\ensuremath{\widetilde{H}}}

\def\NS{\mathop{\mathrm{NS}}\nolimits}

\def\Hdg{\mathrm{Hdg}}
\def\rank{\mathrm{rank}}

\newtheorem{Thm}{Theorem}[section]
\newtheorem{Prop}[Thm]{Proposition}

\newtheorem{Lem}[Thm]{Lemma}
\newtheorem{Cor}[Thm]{Corollary}

\newtheorem{Ques}{Question}

\newtheorem*{Ques*}{Question}

\theoremstyle{definition}
\newtheorem{Def}[Thm]{Definition}
\newtheorem{Rem}[Thm]{Remark}

\newtheorem{Ex}[Thm]{Example}

\makeatletter
\newtheoremstyle{italicsname}
 {3pt}
 {3pt}
 {\itshape}
 {}
 {\itshape}
 {.}
 {.5em}
 {\thmname{#1}\thmnumber{\@ifnotempty{#1}{ }#2}%
 \thmnote{ {\the\thm@notefont(#3)}}}
\makeatother
\theoremstyle{italicsname}

\numberwithin{equation}{section}

\begin{document}

\title[On the transcendental lattices of Hyperk\"ahler manifolds]{On the transcendental lattices of Hyperk\"ahler manifolds}

\author[Benedetta Piroddi, Ángel David Ríos Ortiz]{Benedetta Piroddi,  Ángel David Ríos Ortiz}

\address{Dipartimento di Matematica ``Federigo Enriques'', Universit\`a degli Studi di Milano, Via Cesare Saldini 50, 20133 Milano, IT.}
\email{benedetta.piroddi@unimi.it}
\address{Universit\'e Paris-Saclay, CNRS, Laboratoire de Math\'ematiques d'Orsay, B\^at. 307, 91405 Orsay, France}
\email{angel-david.rios-ortiz@universite-paris-saclay.fr}

\begin{abstract}
We introduce the notion of a Hyper-K\"{a}hler manifold $X$ induced by a Hodge structure of K3-type. We explore this notion for the known deformation types of hyper-K\"{a}hler manifolds studying those that are induced by a K3 or abelian surface, giving lattice-theoretic criteria to decide whether or not they are birational to a moduli space of sheaves over said surface. We highlight the different behaviors we find for the particular class of hyper-K\"{a}hler manifolds of O'Grady type.
\end{abstract}

\maketitle

\section{Introduction}

Ricci flat compact K\"ahler manifolds have a central place in both the study of geometry and the world of theoretical physics. The famous Beauville-Bogomolov decomposition theorem distinguishes three building blocks: complex tori, Calabi-Yau and Hyper-K\"{a}hler (HK) manifolds. This latter class of manifolds will be the main object of the present work. HK manifolds occupy an enigmatic and distinctive position in the classification: enigmatic since there are very few known examples and no reasons to believe these are the only possible ones; distinctive because strong structure theorems hold for them, making this class a testing ground for several conjectures in complex geometry and beyond. In this paper we will look at them from the point of view of Algebraic Geometry and through the lens of Hodge theory.\\

K3 surfaces, i.e. simply connected Ricci-flat complex surfaces, are HK manifolds of dimension two. The celebrated Torelli theorem of Pyatetskii-Shapiro-Shafarevich-Burns-Rapoport  states that two K3 surfaces can be distinguished by their weight-two Hodge structure together with the intersection product. For higher dimensional HK manifolds a similar role is played by the  the weight-two Hodge structure in cohomology together with a lattice structure given by a natural non-degenerate quadratic form, known as BBF-form. An analogue of the Torelli theorem using this structure and some non-Hodge theoretical data was laid out by Verbitsky, see D. Huybrechts' excellent Bourbaki talk  \cite{Huybrechts2012} for a complete account.\\

From the point of view of Hodge theory, the most interesting piece of the weight-two cohomology is given by its \emph{transcendental part} with the lattice structure induced by the BBF-form: this will be the main player in the present paper. The philosophy behind this work is that two (projective) HK manifolds that share the same transcendental lattices are tightly related to each other, even if they are \emph{not} of the same deformation type, and sometimes even if one of them is not a HK manifold but has a Hodge structure that \emph{resembles} one.\\

Indeed we point out that, with a few exceptions, all known constructions of HK manifolds are obtained as some moduli space on either a symplectic surface (abelian or K3) or a cubic fourfold: hence, the transcendental part of their Hodge structure will be completely determined by the symplectic surface or cubic fourfold.\\

With the philosophy explained above, in this work we will investigate the following question: given a hyper-K\"{a}hler manifold belonging to one of the known deformation types for which the transcendental lattice is that of a symplectic surface (we call them \emph{induced}, see Definition \ref{Def:inducedHK}), is this manifold birational to a moduli space over said surface? In Section 3  we positively answer such a question for the two \emph{classical} deformation types in every dimension. On the other hand, we obtain that for the most \emph{exotic} types of HK manifolds, the question turns out to be more delicate and interesting.\\

This paper is structured as follows. Section 2 is devoted to preliminaries. In Section 3 we discuss Beauville's examples \cite{Beauville1983}, the first discovered examples of HK manifolds. We prove that a HK manifold $X$ deformation equivalent to a \emph{generalized Kummer} manifold on an abelian surface $A$ is induced by $A$ if and only if $X$  is birational to a moduli space on $A$, and an analogous result holds if $X$ is birational to a Hilbert scheme of points on a K3 surface $S$ and is induced by $S$
(the latter case was already essentially proved by Markman in \cite{Markman2010}).\\

In Section 4 we study O'Grady's examples \cite{O'Grady1997,O'Grady2006}, where we can actually see new phenomena appearing: transcendental data is not enough to determine whether a manifold will be a (desingularized) moduli space, but we have to include \emph{algebraic} data; this is manifestly shown in the cohomology of these manifolds, and to explore these cases lattice theory is needed. We give lattice-theoretic criteria and construct many examples of Hodge structures of manifolds in O'Grady's families which are induced by symplectic surfaces but are \emph{not} moduli spaces. We highlight the very different behavior of induced HK manifolds which are or aren't moduli spaces, and we point out connections with the non-modular construction of O'Grady's 10 dimensional example due to Laza-Saccà-Voisin.\\

The paper closes with Section 5, dedicated to remarks and open questions. For the sake of completeness in Appendix A all the necessary lattice theory needed in the paper is carefully stated with precise references to the literature.

\section*{Acknowledgments}
One motivation of writing this paper came from a question asked by Kieran O'Grady to the second author back in 2020, we thank Kieran for that insight. Also we thank Annalisa Grossi, Giovanni Mongardi and Claudio Onorati for several illuminating discussions and their useful comments.

\section{Induced Hyperk\"ahler manifolds}

A hyper-K\"{a}hler (HK) manifold  $X$ is a smooth, simply connected, compact K\"{a}hler manifold such that $H^0(X,\Omega_X^2)$ is generated by an everywhere non-degenerate holomorphic two-form $\sigma_X$. The second cohomology group of $X$ has therefore a pure weight-two Hodge structure of the following form:
\begin{equation}
    H^2(X,\C) = H^{2,0}(X)\oplus H^{1,1}(X)\oplus H^{0,2}(X) = \C\sigma_X \oplus H^{1,1}(X)\oplus \C \overline{\sigma}_X.
\end{equation}
The existence of the symplectic form $\sigma_X$ implies that $X$ has trivial canonical class and even complex dimension.

\begin{Thm}[Beauville, Bogomolov, Fujiki]\label{Thm:BeauvilleBogomolovFujikiform}
Let $X$ be a HK manifold of dimension $2n$. There exists a primitive integral quadratic form $q_X$ on $H^2(X,\Z)$ which is non-degenerate, invariant under deformations and of signature $(3, b_2(X)-3)$.
\end{Thm}

The quadratic form $q_X$ in Theorem \ref{Thm:BeauvilleBogomolovFujikiform} is known as the Beauville-Bogomolov form on $X$. The integrality of $q_X$ gives to $H^2(X,\Z)$ the structure of an integral lattice that depends only on the deformation type. This lattice structure has been determined  in \cite{Beauville1983,Rapagnetta2007,Rapagnetta2008} whose results are summarized in Table \ref{Table:BBformknowndeformationtypes}, cf. Appendix \ref{Sec:LatticeTheory} for the relevant notations.

\begin{table}[ht]
\centering
\begin{tabular}{|c|c|c|c|}
\hline
$X$ & ${\rm dim}(X)$ & $b_{2}(X)$ & $(H^{2}(X,\Z)$, $q_{X}$) \\
\hline
$\mathrm{K3}^{[n]}$ & $2n$ & $23$ & $\Lambda_{\mathrm{K3}^{[n]}} := U^{\oplus 3}\oplus E_{8}^{\oplus 2}\oplus\langle -2(n-1)\rangle$ \\
\hline
${\rm Kum}^{n}$ & $2n$ & $7$ & $\Lambda_{\mathrm{Kum}^{n}} := U^{\oplus 3}
\oplus \langle -2(n+1)\rangle$ \\
\hline
$\mathrm{OG6}$ & $6$ & $8$ & $\Lambda_{\mathrm{OG6}}:=U^{\oplus 3}
\oplus\langle-2\rangle^{\oplus 2}$ \\
\hline
$\mathrm{OG10}$ & $10$ & $24$ & $\Lambda_{\mathrm{OG10}}:=U^{\oplus 3}\oplus E_{8}^{\oplus 2}\oplus A_2$ \\
\hline
\end{tabular}
\caption{The Beauville-Bogomolov lattice for the known deformation types of HK manifolds.\label{Table:BBformknowndeformationtypes}}
\end{table}

Since $X$ is simply connected, its Néron-Severi group $\NS(X)$ is a free abelian group and therefore a lattice when is equipped with the quadratic form $q_X$.

\begin{Def}
Let $X$ be a HK manifold. The transcendental lattice of $X$ is the lattice $T(X)= \NS(X)^{\perp q_X}$, endowed with the Beauville-Bogomolov form $q_X$ and the weight-two Hodge structure induced by $H^2(X,\Z)$.
\end{Def}

The transcendental lattice is a \emph{primitive} sublattice of $H^2(X,\Z)$ (see Definition \ref{embedding}). Isomorphisms of transcendental lattices are \emph{Hodge isometries}, that is, isomorphisms of Hodge structures that are isometries with respect to the Beauville-Bogomolov form. Since the Neron-Severi group has trivial Hodge structure (it is identified with integral classes in $H^{1,1}(X)$ by Lefschetz's Theorem) one expects that \emph{most} of the geometric information of the HK manifold is already encoded in its transcendental lattice. Moreover, the transcendental lattice is a \emph{birational invariant}.

\begin{Lem}\cite[(1.6.2)]{O'Grady1997}\label{Lem:bimeromorphicisomorphiclattices}
If two HK manifolds $X$ and $Y$ are bimeromorphic, then there exists a Hodge isometry $H^2(X, \Z)\cong_{\Hdg} H^2(Y,\Z)$. In particular their transcendental lattices are Hodge isometric.
\end{Lem}

Following Charles \cite[Definition 4.1]{Charles2022} we will define an abstract Beauville-Bogomolov form on a Hodge structure of K3-type -- see also \cite[Chapter 4]{Huybrechts2016}.

\begin{Def}
A pure integral Hodge structure $V$ is said to be of \emph{K3-type} if $V$ has weight two and 
\[
\dim_\C(V^{2,0})=1 \qquad \text{and} \qquad V^{p,q} = 0 \quad \text{for} \quad |p-q|>2.
\]
A \emph{Beauville-Bogomolov form} on $V$ is a non-degenerate quadratic form $q$ on $V$ that induces a morphism of integral Hodge structures $q:V\otimes V\to\Z$ and is positive definite on the real part of $V^{2,0}\oplus V^{0,2}$.    
\end{Def}

Two Hodge structures of K3-type endowed with a Beauville-Bogomolov form are isomorphic if there exists an isomorphism of Hodge structures which is an isometry with respect to their Beauville-Bogomolov forms.

\begin{Ex}
If $X$ is a HK manifold, then $H^2(X,\Z)$ endowed with its Beauville-Bogomolov form (see Table \ref{Table:BBformknowndeformationtypes}) is a Hodge structure of K3-type. The transcendental lattice $T(X)$ with the Beauville-Bogomolov form is also a Hodge structure of K3-type.
\end{Ex}
\begin{Ex}
Let $A$ be an abelian surface, then $H^2(A,\Z)$ 
endowed with the intersection form is a
Hodge structure of K3-type whose lattice structure is $H^2(A,\Z)=U^{\oplus 3}$.
The transcendental lattice $T(A)= \NS(A)^\perp$  with the restriction of the intersection form is also a Hodge structure of K3-type.
\end{Ex}
\begin{Ex}\label{Ex:cubicfourfold}
Let $Y\subseteq\P^5$ be a cubic fourfold, $h\in H^2(Y,\mathbb Z)$ the hyperplane class. Then $H^4_{prim}(Y,\Z)={(h^2)}^{\perp_{H^4(Y,\Z)}}$ endowed with the cup product is a
Hodge structure of K3-type: its lattice structure is $H^4_{prim}(Y,\Z)=U^{\oplus 2}\oplus E_8^{\oplus 2}\oplus A_2$.
The transcendental lattice $H^4_{tr}(Y,\Z):= (H^{2,2}Y\cap H^4_{prim}(Y,\Z))^\perp$ with the restriction of the cup product is also a Hodge structure of K3-type.
\end{Ex}

The abstract definition of a Hodge structure of K3-type is sufficiently broad to include other Hodge structures which are not necessarily of geometric origin, as for example, non-commutative K3 surfaces in the sense of Kuznetsov \cite{Kuznetsov2010}.

\begin{Def}\label{Def:inducedHK}
Let $X$ be a projective HK manifold and $(T,q)$ be a Hodge structure of K3-type with a Beauville-Bogomolov form. We say that $X$ is \emph{induced} by $T$ if there exists a Hodge isometry $T(X)\cong_\Hdg T$.
\end{Def}

Projective HK manifolds of dimension $2$ are projective K3 surfaces. By results due to Orlov \cite[Theorem 3.3]{Orlov1997} a projective K3 surface which is induced by the Hodge structure of another K3 surface is derived equivalent to the latter (that is, their derived categories are isomorphic). Mukai showed  that a K3 surface which is induced by the Hodge structure of another K3 surface is a moduli space over the latter \cite{Mukai1987}. This gives a very precise interpretation of induced HK manifolds in dimension $2$.

\begin{Thm}[Orlov, Mukai]\label{Thm:FourierMukaiPartners}
Let $S$ and $S'$ be two projective K3 surfaces. The following are equivalent:
\begin{enumerate}
    \item $S$ is induced by $T(S')$.
    \item $S$ and $S'$ are derived equivalent.
    \item $S$ is isomorphic to a fine two-dimensional moduli space over $S'$.
\end{enumerate}
\end{Thm}

\begin{Rem}
In the literature, two projectice K3 surfaces $S$ and $S'$ are said to be \emph{Fourier-Mukai partners} if any of the equivalent statements of Theorem \ref{Thm:FourierMukaiPartners} hold for $S$ and $S'$.  
\end{Rem}

Let $S$ be a K3 surface, then the set of K3 surfaces induced by $T(S)$ modulo isomorphism is finite \cite[Proposition 16.3.10]{Huybrechts2016}. We can generalize this result to the case of induced HK manifolds.

\begin{Prop}\label{Prop:finitenessinducedHKbir}
Let $(T,q)$ be a fixed Hodge structure of K3-type, and fix a lattice $\Lambda$; then the set 
\[
\{ \text{$X$ HK with $H^2(X,\Z)=\Lambda$, $X$ induced by $(T,q$} \}/\sim_{bir}
\]
is finite.
\end{Prop}
\begin{proof}
Let $X$ be a HK manifold induced by $(T,q)$ and let $\Lambda_X:=H^2(X,\Z)$ be the abstract lattice of $X$. Denote by $\mathrm{Mon}(X)$ the monodromy group of $X$. By \cite[Theorem 3.4]{Verbitsky2013} this is a finite index subgroup of $O(\Lambda_X)$. Consider the set $\mathrm{Emb}(T(X),\Lambda_X)$ of primitive embeddings of abstract lattices. The set of orbits of $\mathrm{Emb}(T(X),\Lambda_X)$ under the action of $\mathrm{Mon}(X)$ is finite. Under a fixed orbit, the lattice $T(X)\oplus \NS(X)$ is a sublattice of $H^2(X,\Z)$ of finite index: this implies that there exists finitely many ways of extending the Hodge structure on $T(X)$ to $\Lambda_X$, and any such determines $X$ up to birational isomorphisms by the Hodge theoretic Torelli theorem \cite[Theorem 1.3]{Markman2011}. 
\end{proof}
\begin{Cor}\label{Cor:finitenessinducedHKiso}
Let $(T,q)$ be a fixed Hodge structure of K3-type, and fix a lattice $\Lambda$ of rank at least 5. Then the set
\[
\{ \text{$X$ HK with $H^2(X, \Z)\cong \Lambda$ and induced by $(T,q)$} \}/\cong_{iso}
\]
is finite.
\end{Cor}
\begin{proof}
By the assumption on the rank of $\Lambda$ it follows by Proposition \ref{Prop:finitenessinducedHKbir} and the positive solution by Amerik and Verbitsky to the Kawamata-Morrison Cone Conjecture for HK manifolds \cite{AmerikVerbitsky2020}.
\end{proof}

A result analogous to Theorem \ref{Thm:FourierMukaiPartners} for HK manifolds is at the moment out of reach: however, if we restrict to those belonging to the known deformation types, and which are moreover induced by a symplectic (abelian or K3) surface, some result can be indeed obtained. This will be done in the following sections.

\section{Induced HKs in Beauville's deformation families}

\subsection{Moduli of sheaves on symplectic surfaces}\label{SubSec:ModuliSheavesonSymplecticSurfaces}
In the following, $S$ will be a projective K3 or abelian surface. Let $\tH(S)$ denote the even cohomology ring, i.e.
\begin{equation}\label{eq:mukailattice}
\tH(S) := H^0(S,\Z)\oplus H^2(S,\Z)\oplus H^4(S,\Z).    
\end{equation}
We define a pure weight-two Hodge structure on $\tH(S)\otimes\C$ by requiring the degree 0 and 4 parts to be algebraic:
\[
\tH^{0,2}(S) := H^{0,2}(S), \quad \tH^{2,0}(S) := H^{2,0}(S), \quad \tH^{1,1}(S) := H^0(S)\oplus H^{1,1}(S)\oplus H^4(S).
\]
For any $v = (v_{0},v_{2},v_{4}) \in \tH(S)$, with degree $q$ component given by $v_q$, we set $v^\vee := (v_0, - v_2, v_4)$. On $\tH(S)$ we define \emph{Mukai's bilinear symmetric form} by
\begin{equation}\label{eq:mukaibilinearform}
\langle u,w \rangle := -\int_S u\wedge w^{\vee} = \int_S u_2\wedge w_2 - \int_S (u_0\wedge w_4 + u_4\wedge w_0).    
\end{equation}
\begin{Def}\label{Def:MukailatticeMukaivector}
The \emph{Mukai lattice} of $S$ is the free module $\tH(S)$ with Mukai's bilinear symmetric form $\langle\cdot,\cdot\rangle$. An element $v = (v_{0},v_{2},v_{4}) \in \tH(S)$ is called a \textit{Mukai vector} if $v_{0}\geq 0$ and $v_{2}\in \NS(S)$.
\end{Def}

As a lattice, $\tH(S)$ is isometric to the \emph{abstract Mukai lattice}

\begin{equation}
    \widetilde{\Lambda_S} :=
    \begin{cases}
    \widetilde{\Lambda_{\mathrm{K3}}} =\Lambda_{\mathrm{K3}}\oplus U & \text{if $S$ is a K3 surface}\\
    \widetilde{\Lambda_{\mathrm{Ab}}} = \Lambda_{\mathrm{Ab}}\oplus U & \text{if $S$ is an abelian surface}
    \end{cases}
\end{equation}
where $\Lambda_{\mathrm{K3}} = E_8^{\oplus 2}\oplus U^{\oplus 3}$ and $\Lambda_{\mathrm{Ab}} := U^{\oplus 3}$; notice that $\tH(S)$ is a Hodge structure of K3-type. 

Let $\mathscr{F}$ be a coherent sheaf on $S$. Define the \textit{Mukai vector} of $\mathscr{F}$ to be 
\begin{equation}
v(\mathscr{F}):=\mathrm{ch}(\mathscr{F})\sqrt{\mathrm{td}(S)}=(rk(\mathscr{F}), c_{1}(\mathscr{F}), \frac{1}{2}(c_1(\mathscr{F})^2-2c_2(\mathscr{F}))+\varepsilon rk(\mathscr{F})), 
\end{equation}
where the last equality is Hirzebruch-Riemann-Roch Theorem, with $\varepsilon=1$ if $S$ is K3, and $\varepsilon=0$ if $S$ is abelian. Notice that the Mukai vector of a coherent sheaf is indeed a Mukai vector in the sense of Definition \ref{Def:MukailatticeMukaivector}.

Let $H$ be a polarization and $v$ a Mukai vector on $S$. We write $M_{v}(S,H)$ (resp. $M_{v}^{s}(S,H)$) for the moduli space of $H$-semistable (resp. $H$-stable) sheaves on $S$ with Mukai vector $v$. If $S$ is abelian, a further construction is necessary: we define $K_{v}(S,H):=\mathrm{Alb}^{-1}(0)$, where $\mathrm{Alb}$ is the Albanese morphism, cf. \cite{Yoshioka01}.

\begin{Ex}\label{Ex:hilbertscheme}
Let $H$ be an ample divisor on $S$ and let $v:= (1,0,1-n)$. Then $S^{[n]}:= M_{v}(S,H)$ is the Hilbert scheme of length $n$ zero-dimensional subschemes on $S$. If $S$ is abelian then $K_{v}(S,H)$ is the generalized Kummer manifold associated with $S$.
\end{Ex}

The moduli spaces constructed above can be singular for two reasons: either the Mukai vector is non-primitive, or the polarization is not \emph{$v$-generic}. 
\begin{Def}\label{def:v-generic} Fix a Mukai vector $v\in \tH(S)$ and let $\mathrm{Amp}_\R(S)$ be the ample cone of $S$. A \emph{$v$-wall} is a hyperplane defined by $W_D:= D^{\perp}\cap\mathrm{Amp}_\R(S) $ where $D\in \NS(S)$ satisfies
\begin{equation}
    \frac{v_0^2}{4}(2v_0v_4 - (v_0-1)v_2\cdot v_2) < D\cdot D < 0.
\end{equation}
A polarization $H$ is called \emph{$v$-generic} if $H$ is not contained in any $v$-wall.\end{Def}
\begin{Rem}
For any chosen $v$ there exists a locally finite union of hyperplanes in $\NS(S)\otimes\mathbb R$, outside of which any polarization is $v$-generic.
\end{Rem}

\begin{Thm}\label{Thm:vprimitivo}[Mukai, O'Grady, Yoshioka] Let $S$ be an abelian or projective K3 surface, $v$ a primitive Mukai vector and $H$ a $v$-generic polarization. Then $M_{v}(S,H)=M_{v}^{s}(S,H)$, and we have the following results:
\begin{enumerate} 
\item if $S$ is a K3 surface and $\langle v,v \rangle\geq 2$, then $M_{v}(S,H)$ is an irreducible symplectic variety of dimension $2n=\langle v,v \rangle+2$, which is deformation equivalent to $S^{[n]}$, the Hilbert scheme of $n$ points on $S$. Moreover, there is a Hodge isometry between $v^{\perp}$ and $H^{2}(M_{v}(S,H),\mathbb{Z})$, where the latter has a lattice structure given by the Beauville-Bogomolov form;
\item if $S$ is abelian and $\langle v,v \rangle\geq 6$, then $K_{v}(S,H)$ is an irreducible symplectic variety of dimension $2n=\langle v,v \rangle-2$, which is deformation equivalent to $K^{n}(S)$, the generalized Kummer manifold on $S$, and there is a Hodge isometry between $v^{\perp}$ and $H^{2}(K_{v}(S,H),\mathbb{Z})$.
\end{enumerate}
In particular, $T((M_{v}(S,H)) \cong_\Hdg T(S)$ and $T((K_{v}(S,H)) \cong_\Hdg T(S)$ as pure Hodge structures of weight 2.
\end{Thm}

\subsection{Induced HKs of \texorpdfstring{K3$^{[n]}$}{K3n}-type}\label{sec:inducedK3type}

Let $X$ be a projective HK manifold of K3$^{[n]}$-type and let $\Lambda_{\text{K3}^{[n]}}$ be the abstract K3$^{[n]}$ lattice. Markman constructed in \cite{Markman2010} a natural $O(\widetilde{\Lambda_{\mathrm{K3}}})$-orbit $i_X$ of primitive isometric embeddings of $H^2(X,\Z)$ in $\widetilde{\Lambda_{\mathrm{K3}}}$ (see the definitions \ref{embedding}, \ref{isometric_embeddings}). This allowed him to prove the following.

\begin{Thm}[\cite{Markman2011}, Corollary 9.9]\label{Thm:MarkmanTorelli}
Let $X$ and $Y$ be two manifolds of $K3^{[n]}$-type. Then $X$ and $Y$ are bimeromorphic if and only if there exists a Hodge-isometry $f:H^2(X,\Z)\to H^2(Y,\Z)$, satisfying $i_X = i_Y\circ f$.
\end{Thm}

In particular, this gives a criterion to check if a HK manifold of K3$^{[n]}$-type is birational to a moduli space of sheaves on a K3 surface, which we will apply to the HKs of K3$^{[n]}$-type that are induced by a K3 surface.

Denote by $\mathrm{Emb}(\Lambda_{\text{K3}^{[n]}},\widetilde{\Lambda_{\mathrm{K3}}})$ the set of isometric embeddings of the lattice $\Lambda_{\text{K3}^{[n]}}$ in $\widetilde{\Lambda_{\mathrm{K3}}}$. In \cite[Lemma 4.3.(i)]{Markman2010} Markman establishes a bijective correspondence between the set
\[
P_n = \{ (r,s)\in\Z^2 \,\,\,\text{coprime such that}\,\,\, -s\geq r > 0\,\,\,\text{and}\,\,\, -rs = n-1\}
\]
and the set of $O(\widetilde{\Lambda_{\mathrm{K3}}})$-orbits in  $\mathrm{Emb}(\Lambda_{\text{K3}^{[n]}},\widetilde{\Lambda_{\mathrm{K3}}})$. This correspondence is given by assigning to each pair $(r,s)$ the embedding $(r,0,s)^\perp$ in $\widetilde{\Lambda_{\mathrm{K3}}}$. The following theorem is essentially proved in \cite{Markman2010}, for convenience we give the proof.

\begin{Thm}\label{Thm:InducedK3n}
Let $X$ be a projective HK of $K3^{[n]}$-type and $S$ a projective K3 surface, then $X$ is induced by $T(S)$ if and only if $X$ is birational to a moduli space $M_v(S,H)$ for some $v\in \tH(S)$ and a $v$-generic polarization $H$.
\end{Thm}
\begin{proof}
To prove the implication from right to left notice that by Theorem \ref{Thm:vprimitivo} for any Mukai vector $v$ and $v$-generic polarization $H$ on $S$  we have that $T(M_v(S,H)) \cong_\Hdg T(S)$. By Lemma \ref{Lem:bimeromorphicisomorphiclattices}, if $X$ is birational to $M_v(S,H)$, then $T(X)\cong_\Hdg T(M_v(S,H))$ and we are done. 

For the other implication, suppose that $T(X)\cong T(S)$ and consider the embedding
\[
\iota: H^2(S^{[n]},\Z)\to \tH(S)
\]
given as $(1,0,1-n)^\perp$ as in Example \ref{Ex:hilbertscheme}. By construction we have that $T(S)^\perp$ in $\tH(S)$ contains the hyperbolic plane $H^0(S)\oplus H^4(S)$, hence by Theorem \ref{Nik_1.14.4} we know that this embedding is unique; therefore there exists an isometry $\phi: \widetilde{\Lambda} \to \tH(S)$ such that $\phi(\iota_X(T(X))) = T(S)$.

Set $\Lambda = \phi(H^2(X))$ and let $v := \Lambda^\perp$, this is a primitive vector that takes values in $\tH(S)_{\mathrm{alg}}$. Possibly by changing sign, we can assume that $v$ is a Mukai vector, hence by Theorem \ref{Thm:vprimitivo} for any $v$-generic ample class  $H$ in $S$  we have $v^\perp \cong H^2(M_v(S,H),\Z)$ as Hodge structures: this means that $\phi$ restricted to $H^2(X,\Z)$ is in fact an isometry of Hodge structures that, by construction, extends to the Mukai lattice. By Theorem \ref{Thm:MarkmanTorelli} we conclude that $X$ and $M_v(S,H)$ are birational.
\end{proof}

\begin{Rem}
In the previous section we claimed that $X$ \emph{is} a moduli space and not just birational to one. Using \cite[Theorem 1.2]{BayerMacri2014}, when $X$ is birational to a moduli space of semistable sheaves, then there exists a Bridgeland stability condition $\sigma$ such that $X$ is \emph{isomorphic} to the moduli space of $\sigma$-stable sheaves.
\end{Rem}

\begin{Cor}\label{Cor:UniquenessInducedK3}
Every HK manifold of K3$^{[n]}$-type with Picard rank $\geq 13$ is induced by a unique K3 surface.
\end{Cor}
\begin{proof}
A result of Morrison \cite[Corollary 2.10]{Morrison84} gives that for every Hodge Structure $T$ of K3-type with $rk(T)\leq 12$ there exists a unique K3 surface $S$ such that $T(S) \cong_\Hdg T$. The result follows by Theorem \ref{Thm:InducedK3n}.
\end{proof}

\subsection{Induced HKs of \texorpdfstring{Kum$^n$}{Kumn}-type}\label{sec:inducedKummertype}

The strategy of proof given in the previous section has an analogue for the Kum$^n$ deformation type. Let $X$ be a projective HK manifold of Kum$^n$-type and let $\Lambda_{\text{Kum}^{n}}$ be the abstract Kum$^n$ lattice.

\begin{Thm}\cite[Theorem 4.3]{Mongardi16}\label{Thm:monodromymongardi}
Let $\mathcal{W}(\Lambda_{\text{Kum}^{n}})$ denote the subgroup of $O^+(\Lambda_{\text{Kum}^{n}})$ consisting of orientation preserving isometries acting as $\pm 1$ on the discriminant $\Lambda_{\text{Kum}^{n}}^\vee/\Lambda_{\text{Kum}^{n}}$. Denote by
\[
\chi:\mathcal{W}(\Lambda_{\text{Kum}^{n}})\to \{1,-1\} 
\]
the associated character. Then $\mathrm{Mon}^2(X)$ consists precisely of orientation preserving isometries $g\in \mathcal{W}(\Lambda_{\text{Kum}^{n}})$ such that $\chi(g)\cdot\det(g) = 1$.
\end{Thm}

We need also an analogue of the monodromy invariant embedding due to Markman. Recall the following result by Wieneck.

\begin{Thm}[\cite{Wieneck18}, Theorem 4.9]
Let $X$ be a HK manifold of Kum$^n$-type, $n\geq 2$. Then there exists a canonical monodromy invariant $O(\widetilde{\Lambda_{\mathrm{Ab}}})$-orbit $\iota_X$ of primitive isometric embeddings of $\Lambda = H^2(X,\Z)$ in the Mukai lattice $\widetilde{\Lambda_{\mathrm{Ab}}}$.
\end{Thm}

\begin{Lem}\label{Lem:embeddings}
There exists a bijective correspondence between the set
\[
Q_n = \{ (r,s) \,\,\,\text{coprime such that}\,\,\, -s\geq r > 0\,\,\,\text{and}\,\,\, rs = n+1\}
\]
and the set of $O(\widetilde{\Lambda_{\mathrm{Ab}}})$ orbits in $O(\Lambda_{\text{Kum}^{n}},\widetilde{\Lambda_{\mathrm{Ab}}})$.
\end{Lem}
\begin{proof}
For every $(r,s)\in Q_n$ define the embedding $\iota_{r,s}:\Lambda_{\text{Kum}^{n}}\to \widetilde{\Lambda_{\mathrm{Ab}}}$ by fixing $U^{\oplus 3}\oplus 0$ and sending the generator of $\langle -2(n+1)\rangle$ to $(r,0,-s)$. Notice that the image is given by $(r,0,s)^\perp$. Suppose that there exists an automorphism $g\in O(\widetilde{\Lambda_{\mathrm{Ab}}})$ such that $\iota_{r,s} = \iota_{r',s'}\circ g$, then this morphism must fix $U^{\oplus 3}\oplus 0$ and send $(r,0,-s)$ to $(r',0,-s')$. Since both vectors are contained in $0\oplus U$, then $g$ has to be given by an automorphism of $U$, but $O(U)\cong (\Z/2\Z)^2$
, therefore $g$ must be the identity.

Let $\phi:\Lambda_{\text{Kum}^{n}}\to \widetilde{\Lambda_{\mathrm{Ab}}}$ be an embedding. By Theorem \ref{Nik_1.14.4} there is a unique $O(\widetilde{\Lambda_{\mathrm{Ab}}})$-orbit of embeddings of $U^{\oplus 3}$ into $\widetilde{\Lambda_{\mathrm{Ab}}}$, hence there exists an isometry $g\in O(\widetilde{\Lambda_{\mathrm{Ab}}})$ such that $\phi(U^{\oplus 3})$ is sent to $U^{\oplus 3}\oplus 0$. Let $v\in \Lambda_{\text{Kum}^{n}}$ be a vector generating the sublattice $\langle -2(n+1)\rangle$, since $\widetilde{\Lambda_{\mathrm{Ab}}} = U^{\oplus 3}\oplus U$, we reduce the problem to classify the possible embeddings of $v$ into $U$; since $O(U)\cong (\Z/2\Z)^2$ we have that this embeddings are classified as pairs in the set $Q_n$.
\end{proof}

\begin{Lem}\cite[Lemma 3]{Shioda78}\label{Lem:isomorphismHodgeAbelianDual}
Let $A$ be an abelian surface. There exists a canonical isomorphism of Hodge structures $g:H^2(A^\vee,\Z)\to H^2(A,\Z)$ such that $\det(g) = -1$.
\end{Lem}

By definition of the Mukai lattice, Lemma \ref{Lem:isomorphismHodgeAbelianDual} implies in particular that there is an isomorphism $\tH(A)\cong_\Hdg \tH(A^\vee)$.

\begin{Lem}\label{Lem:adjust-1}
Let $A$ be an abelian surface and fix a Mukai vector $v\in \tH(A)\cong \tH(A^\vee)$; then there exists a Hodge isometry $\phi: H^2(K_v(A^\vee),\Z)\to H^2(K_v(A),\Z)$ for generic polarizations. Moreover, with respect to the Mukai embedding we have $\det(\phi) = -1$ and $\chi(\phi) = 1$.
\end{Lem}
\begin{proof}
By Lemma \ref{Lem:embeddings} we can assume that $v = (r,0,s)$ for some $r,s$. Then, by Theorem \ref{Thm:vprimitivo} we know that
\[
H^2(K_v(A^\vee),\Z) \cong_\Hdg v^\perp \cong_\Hdg H^2(K_v(A),\Z)
\]
as Hodge structures; observe that due to the specific choice of $v$ we can assume the previous isomorphism to be given as $g$ on $H^2(A,\Z)$ and fixing $(r,0,-s)$. This is a Hodge isometry by construction and $\det(\phi) = -1$; observe that since this restricts to the identity on $(r,0,-s)$, then $\chi(g) = 1$ as claimed.
\end{proof}

\begin{Thm}
Let $X$ be a HK manifold of Kum$^n$-type, then $X$ is induced by $T(A)$ if and only if $X$ is birational to $K_v(A,H)$ or $K_v(A^\vee,H)$ for some $v\in \tH(A)$ and some $v$-generic polarization $H$.
\end{Thm}
\begin{proof}
By Lemma \ref{Lem:bimeromorphicisomorphiclattices} we just need to prove the implication from left to right. Suppose that $T(X)\cong T(A)$ and consider the standard embedding $\iota: H^2(Kum_n(A),\Z)\to \tH(A)$ given as $(1,0,1-n)^\perp$. By construction, we have that $T(A)^\perp$ in $\tH(A)$ contains the hyperbolic plane $H^0(A)\oplus H^4(A)$, by Theorem \ref{Nik_1.14.4} we know that this embedding is unique, hence there exists an isometry $\varphi: \widetilde{\Lambda} \to \tH(A)$ such that $\varphi(\iota_X(T(X))) = T(A)$.

Set $\Lambda = \varphi(H^2(X))$ and $v = \Lambda^\perp$: this is a primitive vector that takes values in $\tH(S)_{\mathrm{alg}}$. Possibly by changing sign, we can assume that $v$ is a Mukai vector; by Theorem \ref{Thm:vprimitivo} for any $v$-generic ample class  $H$ in $S$  we have $v^\perp \cong_\Hdg H^2(M_v(S,H),\Z)$. This means that $\varphi$ restricted to $H^2(X,\Z)$ is in fact an isometry of Hodge structures that, by construction, extends to the Mukai lattice. We already know that $\varphi\in \mathcal{W}(\Lambda)$ (cf. the previous Remark), by Theorem \ref{Thm:monodromymongardi} what we are left to show is that $\chi(\varphi)\det(\varphi) = 1$. If not, by Lemma \ref{Lem:adjust-1} there exists a Hodge isometry $\psi: H^2(K_v(A^\vee),\Z)\to H^2(K_v(A),\Z)$ such that $\det(\psi) = -1$ and $\chi(\psi) = 1$, and composing with $\psi$ we obtain the required equality.
\end{proof}

We provide a result analogous to Corollary \ref{Cor:UniquenessInducedK3} for this deformation type.
\begin{Cor}
Every HK manifold of Kum$^{n}$-type with Picard rank $\geq 4$ is induced by a unique abelian surface or its dual.
\end{Cor}
\begin{proof}
Let $X$ be a HK manifold of Kum$^{n}$-type and of Picard rank at least $4$. Then its transcendental lattice $T(X)$ is of rank at most $3$, therefore it occurs as the transcendental lattice of an abelian surface $A$ by \cite[Corollary 2.6]{Morrison84}; moreover, the embedding $T(A)\subseteq U^{\oplus 3}$ is unique, therefore the Hodge structure of $T(A)$ is unique. By \ref{Lem:isomorphismHodgeAbelianDual} we get that $X$ is induced by either $A$ or $A^\vee$.
\end{proof}

\begin{Rem}
An abelian surface $A$ is isomorphic to its dual if and only if admits a principal polarization, which is to say there exists a class $H\in \NS(A)$ such that $H^2=2$. This happens when $A$ is isomorphic either to the Jacobian of a genus $2$ curve or the product of two elliptic curves. 
\end{Rem}

\section{Induced HKs in O'Grady's deformation families}\label{Sec:OgradyExamples}

\subsection{Hodge Structure for singular symplectic varieties}

Let $X$ be a normal complex variety with rational singularities and $\pi:\widetilde{X}\to X$ be a resolution of singularities. It is known (see eg. \cite[Lemma 3.1]{PeregoRapagnetta2013}) that the pull-back morphism 
\begin{equation}
\pi^*:H^2(X,\Z)\to H^2(\widetilde{X},\Z)
\end{equation}
is injective and $H^2(X,\Z)$ carries a pure Hodge structure of weight two induced by $H^2(\widetilde{X},\Z)$. In this paper we are interested in \emph{singular irreducible symplectic varieties} as used in \cite{BakkerLehn2021}.
\begin{Def}
A symplectic variety $X$ is a normal complex algebraic variety  with a holomorphic symplectic form that extends to any resolution of singularities. We say $X$ is singular irreducible symplectic if it admits a symplectic resolution $\pi:\widetilde{X}\to X$, where $\widetilde{X}$  is a HK manifold.
\end{Def}

Singular irreducible symplectic varieties are very similar to HK manifolds. In particular, they admit a non-degenerate quadratic form on their second cohomology, that is compatible with the Hodge structure and the Beauville-Bogomolov pairing on their resolutions.

\begin{Lem}\cite[Lemma 3.5]{BakkerLehn2021}\label{Lem:BBonISV}
 Let $\pi:\widetilde{X}\to X$ be a symplectic resolution. Then the pull-back map $\pi^*:H^2(X,\Z)\to H^2(\widetilde{X},\Z)$ is injective and this injection is an equality on the transcendental part $T(X) = T(\widetilde{X})$. The restriction of $q_{\widetilde{X}}$ to $H^2(X,\Z)$ is non-degenerate.
\end{Lem}

\subsection{Singular moduli of sheaves on symplectic surfaces}

Let $S$ be a projective K3 or abelian surface. We will be using the same notation as in Section \ref{SubSec:ModuliSheavesonSymplecticSurfaces}. Moduli spaces of stable sheaves with \emph{non-primitive} Mukai vector are always singular. However, through the works of O'Grady \cite{O'Grady1997,O'Grady2006} Lehn and Sorger \cite{LehnSorger2006} and its final form by Perego and Rapagnetta in \cite{PeregoRapagnetta2013,PeregoRapagnetta2014}, we have a clear understanding on when these spaces admit a symplectic resolution.

\begin{Thm}\label{Thm:vnonprimitivo}[O'Grady, Lehn-Sorger, Perego-Rapagnetta] Let $S$ be an abelian or projective K3 surface, $v$ a primitive Mukai vector such that $\langle v, v\rangle=2$ and $H$ a $v$-generic polarization. Then $M_{2v}(S,H)$ (resp. $K_{2v}(S,H)$ if $S$ is abelian) is a singular irreducible symplectic variety and we have the following results:
\begin{enumerate}
\item if $S$ is a K3 surface, then 
its symplectic resolution $\pi:\widetilde{M_{2v}}(S,H)\to M_{2v}(S,H)$ is a HK manifold of OG10-type. There is a Hodge isometry between $v^{\perp}$ and $H^{2}(M_{v}(S,H),\mathbb{Z})$, where the latter has a lattice structure given by Lemma \ref{Lem:BBonISV}. Moreover, if $\alpha\in v^\perp$ has divisibility $2$, then $\frac{\alpha \pm \sigma}{2}\in H^2(\widetilde{M_{2v}}(S,H),\Z)$, where $\sigma$ is the exceptional divisor of $\pi$. 
\item if $S$ is abelian, then its symplectic resolution $\pi:\widetilde{K_{2v}}(S,H)\to K_{2v}(S,H)$ is a HK manifold of OG6-type, and there is a Hodge isometry between $v^{\perp}$ and $H^{2}(K_{2v}(S,H),\mathbb{Z})$. Moreover, $H^2(\widetilde{K}_{2v}(S,H),\Z)\cong_\Hdg H^2(K_{2v}(S,H),\Z) \oplus_\perp \Z\sigma$, where $\sigma$ is the exceptional divisor of $\pi$.
\end{enumerate}
In particular, there are isomorphisms $T(\widetilde{M_{2v}}(S,H)) \cong_\Hdg T(S)$ and $T(\widetilde{K_{2v}}(S,H)) \cong_\Hdg T(S)$.
\end{Thm}

\subsection{Induced HK's of O'Grady 6 type}

Let $X$ be a projective HK manifold of OG6-type. In this case the monodromy group is \emph{maximal} \cite{MongardiRapagnetta2021}. Hence the birational class of a HK manifold of OG6-type is fully determined by the Hodge structure of its Beauville-Bogomolov lattice.

\begin{Thm}\cite[Theorem 1.1]{MongardiRapagnetta2021}\label{Thm:TorelliOG6}
Let $X, Y$ be two HK manifolds of OG6-type. Then $X$ and $Y$ are bimeromorphic if and only if $H^2(X,\Z) \cong_\Hdg H^2(Y,\Z)$.
\end{Thm}

We want to study whether induced OG6-type manifolds are symplectic resolution of a singular moduli space over an abelian surface or not, similarly to Beauville's examples in the previous sections. 
The following theorem follows a similar strategy as in \cite[Theorem 1.1]{Grossi22}.

\begin{Thm}\label{Thm:InducedOG6}
Let $X$ be a projective HK manifold of OG6-type induced by $T(A)$, where $A$ an abelian surface. Then the following are equivalent:
\begin{enumerate}
    \item There exists an algebraic class $\sigma\in \NS(X)$ of square $-2$ and divisibility $2$.
    \item $X$ is birational to a $\widetilde{K_{2v}}(S,H)$ for some $v\in \tH(A)$ and some $v$-generic polarization $H$.
\end{enumerate}
\end{Thm}
\begin{proof}
The second item implies the first by Theorem \ref{Thm:vnonprimitivo}. For the other implication, notice that by Theorem \ref{Nik_1.15.1} there exist exactly two non-isomorphic primitive embeddings of the lattice $\langle-2\rangle$ in $ \Lambda_{\mathrm{OG6}}$: up to isometries of the latter, to realize the first embedding we can choose as generating class $\sigma=u_1-u_2$ (where $\langle u_1,u_2\rangle=U$), which has divisibility 1, while in the other case, $\sigma$ is the generator of one of the orthogonal components $\langle-2\rangle$, so that it has divisibility $2$. 

In the latter case $\NS(X)$ has $\sigma$ as an orthogonal summand. Let $\Lambda':= \sigma^\perp$ be the orthogonal complement of $\sigma$ in $\Lambda_{\mathrm{OG6}}$, then we have
\begin{equation}
    \Lambda' = \Lambda_{\mathrm{Ab}}\oplus \langle -2 \rangle \cong H^2(K_v(A,H),\Z).
\end{equation}
where the last isomorphism is as abstract lattices. If we let $\pi: \Lambda_{\mathrm{OG6}} \to \Lambda'$ be the projection map, then $T(X)$  is isomorphic to its image under $\pi$. Since $T(X)\cong T(A)$, this defines a level $2$ Hodge structure on $H^2(K_v(A,H),\Z)$ that by construction lifts to an isomorphism of level $2$ Hodge structures
\begin{equation}
   H^2(X,\Z)\cong H^2(\widetilde{K}_v(A,H),\Z).
\end{equation}
By Theorem \ref{Thm:TorelliOG6} we have that $X$ is birational to $\widetilde{K}_v(A,H)$.
\end{proof}

In the rest of this section we will study several examples of induced HK of OG6-type which are not desingularizations of moduli spaces. We first notice that there are strong conditions a lattice has to satisfy to be the transcendental lattice of an abelian surface.

\begin{Thm}[\cite{Morrison84}, Theorem 1.6 and Corollary 2.6]\label{AbTrans}
Let $T$ be an even lattice of signature $(2,k)$.
\begin{enumerate}
    \item If $k=0,1$ there is a primitive embedding $T$ in $\Lambda_{\mathrm{Ab}}$.
    \item If $k=2$ there is a primitive embedding $T$ in $\Lambda_{\mathrm{Ab}}$ if and only if $T\cong U\oplus T'$.
    \item If $k=3$ there is a primitive embedding $T$ in $\Lambda_{\mathrm{Ab}}$ if and only if $T\cong U^{\oplus 2}\oplus T'$.
\end{enumerate}
Moreover there exists a (not necessarily unique) abelian surface $A$ such that $T(A)\cong_\Hdg T$.
\end{Thm}

If an induced HK of OG6-type does not satisfy the orthogonality condition in Theorem \ref{Thm:vnonprimitivo}, then it cannot arise from a moduli space. We are going to give examples of this occurrence for each of the cases allowed by Theorem \ref{AbTrans}.

\begin{Ex}\label{controes}
Fix generators $u_1,u_2, v_1,v_2, w_1,w_2,a,b$ of the OG6-lattice $\Lambda_{\mathrm{OG6}}$ where $u_1,...,w_2$ are generators of the hyperbolic planes and $a,b$ are the two classes of square $-2$ and divisibility $2$ (see Table \ref{Table:BBformknowndeformationtypes}).
\begin{enumerate}
\item $\rank(T)=5$: Consider the lattice $T=U\oplus U\oplus\langle-4\rangle$: embed it in $\Lambda_{\mathrm{OG6}}$ as $\langle u_1,u_2,v_1,v_2,\\a+b\rangle$. Then $T^{\perp}=U\oplus\langle-4\rangle$, which does not satisfy the condition of Theorem \ref{Thm:vnonprimitivo}.
Indeed, a primitive class of even divisibility is of the type $(2k+1)(a-b)+2(hw_1+jw_2)$ with $k,h,j$ integers, but none of these classes has square $-2$ (to this end, odd multiples of $w_1,w_2$ are needed).
\item $\rank(T)=4$: Consider $T=U\oplus\langle6\rangle\oplus\langle-10\rangle$: embed it in $\Lambda_{\mathrm{OG6}}$ as $\langle u_1,u_2,2(v_1+v_2)+a,2(w_1-w_2)+b\rangle$; then $T^{\perp}=A_2\oplus\begin{bmatrix}-2 & 3\\ 3 &-2\end{bmatrix}$, which does not admit any primitive embedding of $\langle-2\rangle$ with divisibility 2.
\item $\rank(T)=3$: Consider $T=U\oplus\langle4\rangle$: embed it in $\Lambda_{\mathrm{OG6}}$ as $\langle u_1,u_2,2(v_1+v_2)+a+b\rangle$; then $T^{\perp}=U\oplus A_3$, which does not admit any primitive embedding of $\langle-2\rangle$ with divisibility 2.
\item $\rank(T)=2$: Consider $T=\langle 6\rangle^{\oplus 2}$: embed it in $\Lambda_{\mathrm{OG6}}$ as $\langle 2(u_1+u_2)+a,2(v_1+v_2)+b\rangle$: then $T^{\perp}=U\oplus A_2^{\oplus 2}$, which does not admit any primitive embedding of $\langle-2\rangle$ with divisibility 2.
\end{enumerate} 
\end{Ex}

The following theorem provides a characterization of induced HKs with transcendental $T_d= U\oplus U\oplus \langle-2d\rangle$, $d\in \mathbb N_{\neq 0}$, as in case (3) of Theorem \ref{AbTrans}.

\begin{Thm}
Let $X$ be an induced HK of OG6-type such that $\rank(NS(X))=3$. Then $X$ does not arise from a moduli space of sheaves on an abelian surface if and only if it holds \begin{equation}\label{controes.general1}
d=4k+2, \qquad
T(X)\simeq U\oplus U\oplus \langle-2d\rangle, \qquad \NS(X)\simeq \begin{bmatrix} 2k& 1 &1\\ 1 &-2&0\\1 &0 &-2\end{bmatrix}.
\end{equation}
\end{Thm}
\begin{proof}
The proof relies heavily on Theorem \ref{Nik_1.15.1}. If $d=0,3\ (\mathrm{mod\ }4)$ there exists only one possible embedding of $T_d$ in $\Lambda_{\mathrm{OG6}}$ up to isometries of the latter, given by $T_d\hookrightarrow U^{\oplus 3}\hookrightarrow \Lambda_{\mathrm{OG6}}$: the corresponding $\NS(X)$ is $S_d := \langle 2d\rangle\oplus\langle -2\rangle^{\oplus 2}$, which obviously contains $\sigma$ as in the statement. Moreover, $S_d$ is unique in its genus (see for instance Proposition \ref{MM_VIII.4.2}).\\
If $d=1\ (\mathrm{mod\ }4)$, there are two possible embeddings of $T_d$ in $\Lambda_{\mathrm{OG6}}$: one as above with orthogonal complement $S_d$, the other described by $\langle u_1,u_2,v_1,v_2,\frac{d-1}{2}w_1+\frac{d-1}{2}w_2+\frac{d+1}{2}a\rangle$, which gives as $\NS(X)$ a lattice of the form
$S'\oplus\langle-2\rangle$, with $sign(S')=(1,1)$, so it satisfies the statement. The lattice $\NS(X)$ is again unique in its genus, since its discriminant form is $q_S=[-1/2d]$ so $rk\NS(X)=2+\ell(\NS(X))$ (see Proposition \ref{MM_VIII.4.2}).\\
If $d=2\ (\mathrm{mod\ }4)$ there are again two possible embeddings, one of which gives $\NS(X)=S_d$, and the other described by $\langle u_1,u_2,v_1,v_2,2kw_1-2kw_2+a+b\rangle$, where $k=(d-2)/4$:
in this case it holds
\begin{equation}\label{controes.general2}
\NS(X)=\begin{bmatrix} 2k& 1 &1\\ 1 &-2&0\\1 &0 &-2\end{bmatrix}
\end{equation}
which is unique in its genus (its discriminant form is $[(1-d)/2d]$). This lattice does not contain any classes $\sigma$ that satisfy the requirements: indeed, calling  $\{b_1,b_2,b_3\}$ a basis of $\NS(X)$, to have even divisibility it has to be $\sigma=\sum x_ib_i$ with $x_1$ even, and $x_2=_2x_3$: none of these classes however have square $-2$. 
\end{proof}

\begin{Rem}
Taking $k=0$, choosing the basis $\{b_1,b_2+b_1,b_3-b_2-2b_1\}$ we get $\NS(X)=U\oplus\langle-4\rangle$ as in Example \ref{controes}.
\end{Rem}

\begin{Rem}
Let $d=2$: if $X$ does not arise from a moduli space, then $\NS(X)=U\oplus\langle-4\rangle=\langle u_1,u_2,x\rangle$; the movable cone coincides with the positive cone, because there are no classes of divisibility 2; the walls of the K\"ahler cone are $\alpha^\perp$, for any $\alpha$ of the form $\alpha=au_1+bu_2+cx$ such that $c^2-ab=1$; on the other hand, if $\NS(X)=\langle4\rangle\oplus\langle-2\rangle^{\oplus 2}=\langle y,\alpha_1,\alpha_2\rangle$, then the walls are of the form $ay+b\alpha_1+c\alpha_2$ such that either $b^2+c^2=2a^2+1$, or the following conditions hold: $b,c$ are odd, and $b^2+c^2=2a^2+2$. The movable cone coincides with the ample cone (all classes of square $-2$ have divisibility 2).
\end{Rem}
\begin{Prop}
If $X$ is of OG6-type, $T(X)=U\oplus U\oplus\langle-4\rangle$
 and $\NS(X)=U\oplus\langle-4\rangle$, then $Bir(X)$ is infinite.
\end{Prop}
\begin{proof}
By the solution of the Kawamata-Morrison Cone Conjecture \cite{AmerikVerbitsky2020} we need to prove that there exist an infinite number of walls inside the movable cone. By the previous remark it is therefore sufficient to show that for any class $\alpha$ of the form $\alpha=au_1+bu_2+cx$ such that $c^2-ab=1$ there exists a positive class $\beta=du_1+eu_2+fx$ such that $\alpha\beta=0$: to this end, we need to prove that the system
\[\Sigma: \begin{cases}
af+be-4cg=0\\
2ef-4g>0
\end{cases}\]
admits a solution for any choice of $a,b,c$ such that $c^2-ab=1$.\\
Notice firstly that we can always suppose $a\geq b$. If $c=0$, then $(a,b)=(1,-1)$: then any choice of $e,f,g$ that satisfies $e=f, g<e^2/2$ is a solution to $\Sigma$. If $c\neq 0$, then $a,b$ have the same sign, and we write
\[\Sigma: \begin{cases}
g=(af+be)/4c\\
2ef-(af+be)/c>0;
\end{cases}\]
choose $e,f$ such that $ef>c>0$: then $2efc-(af+be)>2+2ab-af-be$, which is positive for any choice of $e,f$ with opposite sign to $a,b$. Therefore, if $c>0$ a solution to $\Sigma$ is given by any pair $e,f$ with opposite sign to $a,b$, such that  $4c$ divides both $e$ and $f$. If $c<0$, take instead any pair $e,f$ with the same sign as $a,b$, such that  $4c$ divides both $e$ and $f$.
\end{proof}

We will now give some more details in case $\rank(T)=4$: it always holds $T=U\oplus Q$, where 
$Q=\begin{bmatrix}2\alpha &\beta\\ \beta & 2\gamma\end{bmatrix}$ such that $d:= \beta^2-4\alpha\gamma>0$.

\begin{Lem} The following conditions are equivalent:
\begin{enumerate}
\item $d$ is even;
\item $\beta$ is even;
\item $\ell(A_Q)=2$.
\end{enumerate}
\end{Lem}
\begin{proof}
If $d$ is odd, then the discriminant group of the transcendental lattice is $A_T=\mathbb Z/d\mathbb Z$. If $\beta$ is even, calling $q_1,q_2$ the generators of $Q$, then the vectors $q_1/2, q_2/2$ belong to $A_T$, which therefore contains $(\mathbb Z/2\mathbb Z)^2$ as subgroup.
\end{proof}

\begin{Cor} If $\ell(A_Q)=1$, then any HK of OG6-type with transcendental lattice $T=U\oplus Q$ arises from a moduli space.
\end{Cor}

\begin{Ex}
If $Q=\begin{bmatrix}
    2\alpha & 0 \\ 0 & -2\gamma
\end{bmatrix}$, then $T=U\oplus Q$ can be the transcendental lattice of an induced HK of OG6-type that does not arise from a moduli space if and only if either $\gamma=1,\alpha=3\ (\mathrm{mod\ }4)$, or $\gamma=2,\alpha=2\ (\mathrm{mod\ }4)$: in the first case, the embedding can be realized as $Q=\langle 2(v_1+kv_2)+a,2(w_1-hw_2)+b\rangle$, giving $\alpha=4k-1, \gamma=4h+1$; in the second case, $Q=\langle 2(v_1+kv_2)+a+b,2(w_1-hw_2)+a-b\rangle$, giving $\alpha=4k-2, \gamma=4h+2$. The corresponding Néron-Severi lattice will be, respectively, 
\[\begin{bmatrix}
    -2k^2 & 1-2k &0 &0\\
    1-2k & -2 &0 &0\\
    0 &0 &-2h^2 &1+2h\\
    0 &0 &1+2h &-2
\end{bmatrix},\qquad\begin{bmatrix}
    -2k &0 &1 &1\\
    0 &2h &1 &-1\\
    1 &1 &-2 &0\\
    1 &-1 &0 &-2
\end{bmatrix}.\]
Neither contains classes of square $-2$ and divisibility $2$.
\end{Ex}

\subsection{Induced HK's of O'Grady 10 type}

Let $X$ be a projective HK manifold of OG10-type. As in the OG6-type case the monodromy group is \emph{maximal}, therefore the birational class of a HK manifold of OG10-type is fully determined by the Hodge structure of its Beauville-Bogomolov lattice.

\begin{Thm}\cite[Theorem 5.4]{Onorati2022}
Let $X, Y$ be two HK manifolds of OG10-type. Then $X$ and $Y$ are bimeromorphic if and only if $H^2(X,\Z) \cong_\Hdg H^2(Y,\Z)$.
\end{Thm}

We will again follow the strategy of \cite[Theorem 1.1]{Grossi22} in order to find a necessary and sufficient criterion to decide whether a HK of OG10-type  induced by a K3 surface is the desingularization of a moduli space. See also \cite{FGG2023} for a similar result.

\begin{Thm}\label{Thm:InducedOG10}
    Let $X$ be a projective HK manifold of OG10-type which is induced by $T(S)$, where $S$ is a K3 surface. Then the following are equivalent:
    \begin{enumerate}
        \item There exists a class $\sigma\in \NS(X)$ such that $q_X(\sigma) = -6$ and $div(\sigma) = 3$.
        \item $X$ is birational to a $\widetilde{M_{2v}}(S,H)$ for some $v\in \tH(S)$ and some $v$-generic polarization $H$.
    \end{enumerate}
\end{Thm}
\begin{proof}
By Theorem \ref{Nik_1.15.1} there exist exactly two non-isomorphic primitive embeddings of the lattice $\langle -6\rangle$ in $\Lambda_{\mathrm{OG10}}$ up to isometries of the latter: to realize the first embedding we can choose as generating class $u_1-3u_2$ (where $\langle u_1,u_2\rangle=U$), which has divisibility 1; for the other case, take instead as generating class $a_1+2a_2$ (where $\langle a_1,a_2\rangle=A_2$), which has divisibility 3. 

Let $\sigma$ be a class of square $-6$ and divisibility 3 in $\Lambda_{\mathrm{OG10}}$ and let $\Lambda':= \sigma^\perp$ be its orthogonal complement  in $\Lambda_{\mathrm{OG10}}$: then we have
$\Lambda' = U^{\oplus 3}\oplus E_8^{\oplus  2} \oplus \langle -2 \rangle$. The lattice $\Lambda'$ embeds in the Mukai lattice $\tH(S)$ as $\lambda^\perp$, with $\lambda^2=2$, so $\Lambda'\simeq_{\Hdg}H^2(M_{2\lambda}(S,H),\Z)$, for any $\lambda$-general polarization $H$.

Now, if we let $\pi: \Lambda_{\mathrm{OG10}} \to \Lambda'$ be the projection map, then $T(X)$  is isomorphic to its image under $\pi$. Since $T(X)\cong T(S)$, this defines a level $2$ Hodge structure on $H^2(M_{2\lambda}(S,H),\Z)$. We can extend this to an isomorphism of level $2$ Hodge structures
\begin{equation}
   \Lambda_{\mathrm{OG10}}\cong H^2(\widetilde{M}_{2\lambda}(S,H),\Z),
\end{equation}
via the construction of the lattice $A_2$ as overlattice of $\langle-2\rangle\oplus\langle-6\rangle$, which is the only way to embed $\Lambda'$ in $\Lambda_{\mathrm{OG10}}$ primitively: calling $\alpha$ (respectively $\beta$) the generator of $\langle-2\rangle$ (resp. $\langle-6\rangle$), it holds $A_2=\langle \alpha, (\alpha+\beta)/2\rangle$. 
By the strong form of Torelli's theorem we have therefore that $X$ is birational to $\widetilde{M}_{2\lambda}(S,H)$.
\end{proof}

\begin{Ex}\label{smallestOG10}
Let $T=\begin{bmatrix}4 & 2\\ 2 & 4\end{bmatrix}$: by the Torelli theorem there exists a K3 surface $S$ whose transcendental lattice is $T$ (its Néron Severi lattice is $U\oplus E_8\oplus E_6\oplus D_4$). Embed $T$ in $\Lambda_{\mathrm{OG10}}$ as follows: call $\langle u_1,u_2,v_1,v_2,a_1,a_2\rangle$ the standard basis of $U\oplus U\oplus A_2$, then $T=\langle u_1+2u_2-v_1-v_2-a_1,\ u_2-u_1-2v_1-2v_2-a_1-a_2\rangle$; its orthogonal complement in $\Lambda_{\mathrm{OG10}}$ is the lattice $N=U\oplus E_8^{\oplus 2}\oplus D_4$. If $X$ is such that $(NS(X),T(X))=(N,T)$, then $X$ is induced, but it is not birational to $\widetilde{M_{2v}}(S,H)$ for any suitable $v,H$: moreover, the discriminant group of the Néron-Severi lattice of the latter is bigger. 
\end{Ex}

HK manifolds as the one constructed in the example above contain the hyperbolic plane in their Neron-Severi lattice. One important family of OG10-type manifolds which also contains a copy of the hyperbolic plane was produced by Laza, Saccà and Voisin in \cite{LSV2017}   via a compactification of the intermediate Jacobians of hyperplane sections on a general cubic fourfold. Later this construction was further extended by Saccà in \cite{Sacca2020} to include every smooth cubic fourfold. More precisely, if $Y\subseteq \P^5$ is a smooth cubic fourfold, and $J_V\to V\subset (\P^5)^\vee$ the fibration in intermediate Jacobians of smooth hyperplane sections, in \cite[Theorem 1.6]{Sacca2020} is proven that there exists a smooth HK compactification of $J_V$.

\begin{Def}
A HK manifold of OG10-type birational to a smooth compactification of the fibration $J_V\to V\subset (\P^5)^\vee$ for some smooth cubic fourfold $Y\subset\P^5$ as considered above is called an \emph{LSV manifold}.
\end{Def}

\begin{Rem}
The HK compactification of the intermediate Jacobian fibration associated with a smooth cubic fourfold $Y$ is not unique, but its birational class it is. We will use $J_Y$ to denote any compactification of $J_V\to V$.
\end{Rem}

\begin{Rem}The algebraic copy of the hyperbolic plane $U$ in LSV manifolds comes from the construction as follows. There are always two distinguished algebraic classes, an isotropic class $F$ coming from the naturally associated fibration, and a rigid class $\theta$ given by the compactification of the Theta divisor on the fibers: hence, it holds \[\langle F,\theta\rangle=\begin{bmatrix}
    0 &1\\ 1 &-2
\end{bmatrix}.\]
Therefore $U=\langle F,\theta+F\rangle$.\end{Rem}

\begin{Prop}\label{Prop:UthenLSV}
Let $X$ be a HK manifold of OG10-type. If $X$ is an LSV manifold then  $U\subseteq \NS(X)$. Conversely, if $U\subseteq \NS(X)$ and $X$ is very general, then there exists a cubic fourfold $Y$ such that $X$ is birational to $J_Y$.
\end{Prop}
\begin{proof}
If $X$ is an LSV manifold, then is birational to $J_Y$, where $Y$ is a cubic fourfold. By \cite[Lemma 3.5]{Sacca2020} it holds $U\subseteq \NS(J_Y)=\NS(X)$, the last equality by Lemma \ref{Lem:bimeromorphicisomorphiclattices}.

HK manifolds of OG10-type which contain a copy of $U$ in the Neron-Severi form an irreducible moduli space $\mathcal M_U$ of dimension 20, see \cite{Camere2016} for more details on lattice-polarized HK manifolds. Hence, by counting dimensions, in order to prove the claim it suffices to show that LSV manifolds are dense inside $\mathcal M_U$. Let $Y$ be a smooth cubic fourfold, then by \cite[Lemma 3.2]{Sacca2020} there exists an isomorphism
\begin{equation}\label{eq:transcendentalHScubicfourfold}
T(J_Y)\otimes \Q \cong_\Hdg H^4(Y,\Q)_{tr}   
\end{equation}
of rational Hodge structures. 

Let $\mathcal{M} $ be the moduli space of smooth cubic fourfolds (which has again dimension 20) and define the subspace
\[
\mathcal{V} = \{ Y\in \mathcal{M} \,\,\text{such that}\,\, H^{2,2}(Y) = \langle h^2\rangle\}
\]
of very general cubic fourfolds (see Example \ref{Ex:cubicfourfold} for the notation): this is an open subset in $\mathcal{M}$.
 The Torelli Theorem for cubic fourfolds \cite{Voisin1986} states that two cubic fourfolds $Y,Y'$ are isomorphic if and only if there exists a Hodge isometry 
$H^4(Y,\Z)\cong_\Hdg H^4(Y',\Z)$ preserving the square of the hyperplane class. If $Y$ is in $\mathcal{V}$, then ${( h^2)}^\perp=H^4(Y,\Z)_{tr}=H^4(Y,\Z)_{prim}$ which has a fixed lattice structure: therefore the only distinction between cubic fourfolds in $\mathcal V$ is given by the Hodge structure we put on the abstract primitive lattice $H^4_{prim}=U^{\oplus 2}\oplus E_8^{\oplus 2}\oplus A_2$. 
By \eqref{eq:transcendentalHScubicfourfold} this shows the claim. 
\end{proof}

\begin{Rem}
Although Proposition \ref{Prop:UthenLSV} implies that an open dense set of $U$-polarized manifolds of OG10-type $X$ is actually in the LSV family, it is not \emph{explicit} in any way: this prevents us from finding a cubic fourfold $Y$ such that $J_Y$ is birational to a given $X$. We will come back to this problem in Section \ref{Sec:RemarksandQuestions}.
\end{Rem}

Proposition \ref{Prop:UthenLSV} prompts a question: if $X$ is induced by a K3 surface and is not a moduli space, does it always belong to the LSV family? We negatively answer the question in the case of Picard rank 3, which is the most generic case.

 \begin{Prop}
 Let $S$ be a projective K3 surface, $\NS(S)=\langle2d\rangle$: then there exists a HK of OG10-type $X$ which is induced by $T(S)$, but it is not the desingularization of a moduli space, if and only if $d=3(3h+1)$. Moreover, $X$ belongs to the LSV family if and only if $d$ is odd and no prime of the form $6n+5$ divides it. 
 \end{Prop}
 \begin{proof}
The condition $d=3(3h+1)$ ensures by Theorem \ref{Nik_1.15.1} that there exist two different embeddings of the transcendental lattice $T=T(S)$ in $\Lambda_{\mathrm{OG10}}$: indeed if $\gamma$ is a generator of $A_{T}$, then the group $G=\langle\frac{2d}{3}\gamma\rangle$ is isomorphic to the discriminant group of $A_2$. We then compute the orthogonal complement to $G$ in $A_T$, which is $H=\langle3\gamma\rangle$: following Nikulin, if $X$ is not a moduli space, then the discriminant form $q$ of $\NS(X)$ is the opposite to that of $H$, that is, $q=[3/2(3h+1)]$. The lattice $U\oplus\langle -2(3h+1)\rangle$ is unique in its genus (see Proposition \ref{Nik_1.13.4}), and its discriminant form is $\tilde q=[-1/2(3h+1)]$: therefore we only need to find conditions under which $\tilde q$ is equivalent to $q$. \\ Recall that two quadratic forms defined on a finite abelian group $G$ are equivalent if and only if
they are $p$-equivalent for every prime $p$ (see Definition \ref{genus}). 
In our case, since $G=\mathbb Z/2(3h+1)\mathbb Z$ is cyclic, 
we can use (\cite{MM}, Lemma IV.1.4).
Assume that $|A_2|=2^n$: then, it holds $q|_{A_2}=3m/2^n$, and $\tilde q|_{A_2}=-m/2^n$ (for some $m$ odd), which are equivalent if and only if $n=1$. Therefore $3h+1$ should be odd.
Suppose now $p\neq 2$: then $q|_{A_p}=\tilde q|_{A_p}$ if and only if $-2, 6$ are both square, or both non-square modulo $p$. Since $-2$ is square if and only if $p=_8 1,3$, and respectively $-6$ for $p=_{24} 1,5,19,23$, it holds $q|_{A_p}= \tilde q|_{A_p}$ if and only if $p=_{24} 1,7,13,15,19,21$. Taking the complementary (and excluding $p=3$ that does not divide $3h+1$), we get the statement.
\end{proof}
 
\begin{Ex}
Let $\{e_1,\dots e_8\}$ be a $\mathbb Z$-basis of $E_8$ as in Example \ref{ADE}, and let $\{u_1,u_2\}$ be a basis of $U$ as in Remark \ref{U}. Let $S_0$ be a K3 surface such that $\NS(S_0)=\langle 6\rangle$: let $X_0$ be a HK of OG10-type such that $\NS(X_0)$ is generated as a sublattice of $\Lambda_{\mathrm{OG10}}$ by $\langle e_6, u_1, u_2\rangle$: then $X_0$ is induced by $T(S_0)$, it is not a resolution of a moduli space, but it is a member of the LSV family. Indeed, it holds
$\NS(X_0)=U\oplus\langle-2\rangle$.\\
Let $S_1$ be a K3 surface such that $\NS(S_1)=\langle 24\rangle$; let $X_1$ be a HK of OG10-type such that $\NS(X_1)$ is generated as a sublattice of $\Lambda_{\mathrm{OG10}}$ by $\langle e_6,e_7,-e_5+e_8+u_1+3u_2\rangle$: then $X_1$ is induced by $T(S_1)$, but it's neither a resolution of a moduli space, nor a member of the LSV family. Indeed, it holds
 \[\NS(X_1)=\begin{bmatrix}-2 &1 &-1\\ 1 &-2 &1\\-1 &1 &2\end{bmatrix}.\]
 \end{Ex}

\section{Final Remarks and Questions}\label{Sec:RemarksandQuestions}
\subsection{The most algebraic HK manifolds}
K3 surfaces for which the Néron-Severi lattice has maximum rank 20, the so-called \emph{singular} K3 surfaces in the terminology of Shioda-Inose \cite{ShiodaInose1977}, have interesting and rich geometric properties. These surfaces have no moduli and  in loc.cit. a complete classification was given.

\begin{Thm}\cite[Theorem 4.4]{ShiodaInose1977}\label{Thm:ShiodaInose}
Let $\mathcal Q$ be the set of matrices with integral entries of the form $Q=\begin{bmatrix}2a &b\\ b & 2c\end{bmatrix}$ such that $a,c>0,b^2-4ac<0$: there is a bijection between the set of singular K3 surfaces and $\mathcal Q/SL_2(\mathbb Z)$, given by $S\mapsto T(S)$.
\end{Thm}

A natural question is to extend Shioda and Inose's result to HK manifolds, at least for the known deformation types. In the case of the two infinite families introduced by Beauville we get the following.

\begin{Cor}\label{Cor:singularHKBeauville}
Let $X$ be a HK manifold of K3$^{[n]}$-type (resp. Kum$^n$-type) of Picard rank $21$  (resp. $6$). Then $X$ is birational to a moduli space of the unique K3 surface (resp. abelian surface) with transcendental lattice $T(X)$.
\end{Cor}
\begin{proof}
We will prove it for $X$ of K3$^{[n]}$-type, the case of Kum$^n$-type being analogous. The transcendental lattice $T(X)$ is a positive definite lattice of rank $2$, hence it is the transcendental lattice of a unique K3 surface $S$ by Theorem \ref{Thm:ShiodaInose}. By Theorem \ref{Thm:InducedK3n} we get that $X$ is birational to a moduli space of sheaves over $S$.
\end{proof}

However in O'Grady examples, due to the classification results of induced HK given in Section \ref{Sec:OgradyExamples}, there is no analogous result as Corollary \ref{Cor:singularHKBeauville}. Instead, we get the following.

\begin{Cor}
    The transcendental lattice of rank $2$ and smallest discriminant group for which an induced HK manifold is not the desingularized Albanese fiber of a moduli space of sheaves over a K3 surface or an abelian surface) is
    \begin{enumerate}
        \item $\begin{bmatrix}4 &2\\ 2 & 6\end{bmatrix}$ for OG6-type.
        \item $\begin{bmatrix}4 &2\\ 2 & 4\end{bmatrix}$ for OG10-type (see Example \ref{smallestOG10}).
    \end{enumerate}
\end{Cor}
\begin{proof}
Follows by the classification of reduced positive definite binary forms of small determinant  \cite[Table 15.1, pp. 360]{ConwaySloane1999} and a direct computation using Theorems \ref{Thm:InducedOG6} and \ref{Thm:InducedOG10}.
\end{proof}

\subsection{HK manifolds induced by cubic fourfolds}

As already introduced in Example \ref{Ex:cubicfourfold}, another Hodge structure of K3-type worth of consideration is $H^4_{prim}(Y,\Z)_{tr}$, for a smooth cubic fourfold $Y$. Moduli spaces of Bridgeland stable sheaves on a cubic fourfold give rise to HK manifolds of K3$^{[n]}$ type \cite{BLMNPS} or OG10-type \cite{LPZ2022}. Moreover, the compactified intermediate Jacobian fibration \cite{Sacca2020} gives yet another manifold of OG10-type which is \emph{not} a moduli space, at least generically.

\begin{Ques}
Are there analogous Theorems as \ref{Thm:InducedK3n} and \ref{Thm:InducedOG10} for HK manifolds induced by a cubic fourfold?
\end{Ques}

A major obstacle in dealing with the OG10 deformation type is that, at the moment, the Laza-Saccà-Voisin construction associated to a cubic fourfold gives an isomorphism of the transcendental lattice of the OG10-type manifold with the transcendental lattice of the cubic fourfold \emph{only over $\Q$}. 

\subsection{Derived equivalence of induced HK of \texorpdfstring{K3$^{[n]}$}{K3n}-type}

With the aim of generalizing Theorem \ref{Thm:FourierMukaiPartners} for higher dimensional HK manifolds, a natural starting point is the K3$^{[n]}$ deformation type. A specific question to ask is the following.

\begin{Ques}\label{Ques:derivedequivalent}
    Are two HK manifolds of K3$^{[n]}$-type induced by the same K3 surface $S$ derived equivalent?
\end{Ques}

The question above has an affirmative answer in the case of Hilbert schemes of points of K3 surfaces or compactified Jacobians of curves on general K3 surfaces, see \cite[Proposition 1.5]{MMY2020} and \cite[Theorem B]{ADM2016}. In \cite{Beckmann2023} we find a necessary condition for the existence of a derived equivalence. To explain the theorem we need to introduce some notation.

\begin{Def}\cite[Definition 5.2]{Beckmann2023} Let $X$ be a HK manifold of $K3^{[n]}$-type. The integral extended Mukai lattice is the lattice
$\tilde H(X,\Z) := \Z\alpha\oplus H^2(X,\Z)\oplus \Z\beta\subset\tilde H(X,\Q)$. Here $\alpha, \beta$ have self-intersection 0 and $b(\alpha,\beta)=-1$, therefore $\tilde H(X,\Z)$ is isometric to $H^2(X,\Z)\oplus U$.
\end{Def}

Notice that, in contrast with the Mukai lattice introduced in Definition \ref{Def:MukailatticeMukaivector}, the integral extended Mukai lattice is \emph{not} unimodular.

\begin{Def}\cite[Definition 5.3]{Beckmann2023} Let $X$ be a HK manifold of $K3^{[n]}$-type. Let $\delta$ be any class of square $2-2n$ and divisibility $2n-2$ in 
$H^2(X,\Z)$, so that $H^2(X,\Z)\simeq H^2(S,\Z)\oplus \Z\delta$ for some K3 surface $S$. Define the \emph{extended K3$^{[n]}$ lattice} as $\Lambda_X := B_{-\delta/2}(\tilde{H}(X,\Z)\subset\tilde{H}(X,\Q)$, where $B_\lambda\in O(\tilde H(X,\Q))$ is defined as $B_\lambda(r\alpha+\mu+s\beta)=r(\alpha+\lambda)+\mu+(s+b(\lambda,\mu)+r\frac{b(\lambda,\lambda)}{2})\beta$. The definition of $\Lambda_X$ is independent of the choice of $\delta$.
\end{Def}

\begin{Thm}\cite[Theorem 9.2]{Beckmann2023}\label{Thm:DerivedEquivalenceExtendedK3nLattices} Let $X$ and $Y$ be HK manifolds of $K3^{[n]}$-type  and
$\Phi: D^b(X) \simeq D^b(Y)$ a derived equivalence. Then $\Phi^{\tilde H}$ restricts to a Hodge isometry
$\Phi^{\tilde H}: \Lambda_X\simeq \Lambda_Y$ of the extended $K3^{[n]}$-lattices. In particular it induces an isometry between their transcendental lattices.
\end{Thm}

We will show below that in the case of induced HK manifolds, the obstruction given by Theorem \ref{Thm:DerivedEquivalenceExtendedK3nLattices} is not enough to show that two induced HK are not derived equivalent, in particular we can prove the following. 

\begin{Prop}
Suppose $X$ and $Y$ are HK manifolds of $K3^{[n]}$-type induced by the same K3 surface $S$. Then the extended $K3^{[n]}$-lattices $\Lambda_X$ and $\Lambda_Y$ are always Hodge isometric.
\end{Prop}
\begin{proof}
If $X$ is induced, then $T(X)=T(S)$: therefore, $T(X)$ cannot contain any primitive class $\delta$ of square $-2n+2$ and divisibility $2n-2$ (see Definition \ref{def:divisibility}). Any such class is therefore of the form $\delta=(2n-2)t+k\delta'$, with $t\in T(X)$, $\mathrm{gcd}(k,2n-2)=1$ and $ \delta'\in \NS(X)$ of the same square and divisibility of $\delta$, and there exists an isometry $\NS(X)\simeq \NS(S)\oplus \Z\delta'$. Notice that if $S$ admits Fourier-Mukai partners, \emph{any} of them can be chosen in its place (the same holds for $Y$). However $\Lambda_X$ contains an extra algebraic copy of $U$, and if $S,S'$ are FM-partners, even if they are not isomorphic, still $\NS(S)\oplus U$ and $\NS(S')\oplus U$ are isometric (see Prop. \ref{Nik_1.13.4}). Therefore there always exists a Hodge isometry between $\Lambda_X$ and $\Lambda_Y$.
\end{proof}

We can interpret the proposition above as an indication that Question \ref{Ques:derivedequivalent} might have an affirmative answer. A variant of this question was already asked by Markman, cf. \cite[Question 10.8]{Markman2011}.

\appendix
\section{Lattice Theory}\label{Sec:LatticeTheory}

\subsection{Lattices and discriminant forms}

\begin{Def}
An \emph{even lattice} is a free $\mathbb Z$-module $S$ of finite rank, equipped with an integral nondegenerate bilinear symmetric even form $b:S\times S\rightarrow\mathbb Z$ (and an associated nondegenerate quadratic form $q: S\rightarrow 2\mathbb Z$). \\
An even lattice is called \emph{unimodular} if the Gram matrix of $S$ has determinant $\pm 1$.\\
We denote $S(k)$ the lattice with bilinear form $kb$.
\end{Def}

\begin{Def}
An \emph{isomorphism} between lattices (or \emph{isometry}) is an isomorphism of $\mathbb Z$-modules that preserves the bilinear form; the group of isometries of a lattice $S$ is denoted $O(S)$.
\end{Def}

\begin{Ex}\label{ADE}
The $ADE$ lattices are the negative definite even lattices whose Gram matrix is the Cartan matrix of the
Dynkin $ADE$ diagrams. 
\begin{gather*}
\xymatrix@C=0.5pc@R=0.5pc{
&A_n:\quad a_1 \ar@{-}[r] &a_2 \ar@{-}[r] &\dots\ar@{-}[r] &a_n\qquad &D_n:\quad d_2 \ar@{-}[r] & d_3 \ar@{-}[d] \ar@{-}[r]  & \dots \ar@{-}[r] &d_n \\
&\ &\ &\ &\ &\ & d_1 &\ &\ }\\
\xymatrix@C=0.5pc@R=0.5pc{ &E_n: \quad e_2 \ar@{-}[r] &e_3\ar@{-}[r] &e_4 \ar@{-}[r] \ar@{-}[d]  &e_5 \ar@{-}[r] &e_6 \ar@{-}[r] &\dots \ar@{-}[r] &e_n,\\
&\ &\ &e_1 &\ &\ &\ &\ }\end{gather*}
\end{Ex}
\begin{Ex}
The lattice $U$ is the unique even unimodular lattice of rank 2. The lattice $E_8$ is the unique even negative definite unimodular lattice of rank 8.
The K3 Lattice 
$\Lambda_{\mathrm{K3}}\simeq E_8^{\oplus 2} \oplus U^{\oplus 3}$ is the unique even unimodular lattice of signature $(3,19)$.
\end{Ex}

\begin{Rem}\label{U}
The lattice $U$ represents primitively every even number, i.e. for every $n\in\mathbb Z$ we can find a a primitive element $v$ in $U$ with self-intersection $2n$: calling $u_1,u_2$ the generators of $U$ such that their intersection matrix is
\[U=\begin{bmatrix}0 & 1\\ 1 & 0\end{bmatrix},\]
then $v=u_1+nu_2$.
\end{Rem}

\begin{Def}\label{def:divisibility}
Let $s$ be a primitive element of $S$: $s$ has divisibility $d$ if $d|b(s,x)$ for every $x$ in $S$. Notice that, if $S$ is unimodular, $s$ has always divisibility 1.
\end{Def}

\begin{Def} For an even lattice $S$ we can introduce the \emph{dual lattice} $S^*$ as
\[S^*=\{\ x \in S\otimes \mathbb Q \ \mid\ \forall s \in S,\ b_{\mathbb Q}(x, s) \in\mathbb Z\ \}\]
where $b_{\mathbb Q}$ denotes the $\mathbb Q$-linear extension of $b$.\end{Def}
\begin{Rem} 
Notice that $S^*\simeq Hom(S,\mathbb Z)$ via $x\mapsto b_{\mathbb Q}(x, - )$; viceversa, given $f\in Hom(S, \mathbb Z)$, and fixing a basis of $S$ such that $b$ is represented by an integer symmetric nondegenerate matrix $B$, and $f$ by an $1\times n$ vector, the element $x_f$ such that $f=b_{\mathbb Q}(x_f, - )$ is $x_f=fB^{-1}$.
\end{Rem}

\begin{Def}\label{discriminant}
Denote \emph{discriminant group} of $S$ $A_S:= S^*/S$, where $S\hookrightarrow S^*$ via $s\mapsto b(s, - )$. This is a finite group, whose order is equal to $|det(B)|$, where $B$ is any matrix that represents $b$.\\
Denote $\ell$ the \emph{length} of $A_S$, that is defined as its minimum number of generators; denote also $\ell(S)=\ell(A_S)$. 
Define the \emph{discriminant (quadratic) form} $q_S: A_S\rightarrow\mathbb{Q}/2\mathbb{Z}$, induced on $A_S$ by the quadratic form $q$ of $S$.
\end{Def}

\begin{Rem} \label{split}
Given two even lattices $S$ and $K$, it holds $A_{S\oplus K}=A_S\oplus A_K$: in fact $Hom(S\oplus K,\mathbb Z)=Hom(S,\mathbb Z)\oplus Hom(K,\mathbb Z)$, and $(Hom(S,\mathbb Z)\oplus Hom(K,\mathbb Z))/(S\oplus K)$ splits in the direct sum of the two quotients in virtue of the fact that the embedding $S\oplus K\hookrightarrow Hom(S,\mathbb Z)\oplus Hom(K,\mathbb Z)$ is by components.
\end{Rem}

\begin{Def}\label{genus}
The \emph{genus} of a lattice $S$ is the set of all lattices 
with the same signature of $S$ and discriminant form equivalent to $q_S$; notice that lattices in the same genus may not be isomorphic.\\
Two quadratic forms defined on a finite abelian group $G$ are equivalent if and only if they are $p$-equivalent, i.e. equivalent when  restricted to the maximal $p$ group $A_p$ contained in $G$, for every prime $p$.
\end{Def}

The criteria for $p$-equivalence of torsion quadratic forms are given in (\cite{Nikulin79}, \S 1.7-9), and in (\cite{MM}, \S IV). In particular, we refer the reader to Lemma 1.4, Cor. 2.5, Prop. 3.2 of the latter.\\
We write down some sufficient conditions we used above to show that a lattice is unique (up to isometries) in its genus:

\begin{Prop}[\cite{MM}, Cor. VIII.4.2]\label{MM_VIII.4.2}. Let $S$ be an indefinite integral quadratic form with
rank at least 3 and $rk(S) \geq \ell(S) + 2$. Then $S$ is uniquely determined by its signature and discriminant form.\end{Prop}

\begin{Prop}[\cite{Nikulin79}, Cor. 1.13.4]\label{Nik_1.13.4} Let $S$ be an even lattice of signature $(s_+,s_-)$ and discriminant form $q_S$. Then $U\oplus S$ is the unique even lattice with
the same discriminant form $q_S$ and signature $(s_++1,s_-+1)$.
\end{Prop}

\begin{Prop}[\cite{MM}, Cor. VIII.7.8]\label{MM_VIII.7.8}
Let $S$ be an indefinite lattice such that $rk(S)\geq 3$. Write $A_S=\mathbb Z/d_1\mathbb Z\oplus,\dots ,\oplus\mathbb Z/d_r\mathbb Z$
with $d_i\geq 1$ and $d_i\mid d_{i+1}$. Suppose that one of the following holds:
\begin{enumerate}
\item $d_1=d_2=2$;
\item $d_1=2, d_2=4$ and $d_3=_8 4$;
\item $d_1=d_2=4$.
\end{enumerate}
Then $S$ is uniquely determined by its signature and discriminant form (and the map $O(S)\rightarrow O(q_S)$ is surjective).
\end{Prop}

\subsection{Embeddings}

\begin{Def}\label{embedding}
An \emph{embedding} of (even) lattices $(S,q)\hookrightarrow (M,\tilde q)$ is an injective homomorphism of $\mathbb{Z}$-modules such that $\tilde q|_S=q$. In this case, we say that $M$ is an \emph{overlattice} of $S$, or that $S$ is a \emph{sublattice} of $M$. An embedding is \emph{primitive} if $M/S$ is free; an embedding is \emph{of finite index} if $M/S$ is a finite (abelian) group.
\end{Def}

\begin{Def}\label{isometric_embeddings}
Two embeddings $S\hookrightarrow M$,  $S\hookrightarrow M'$ are \emph{isomorphic} if there is an isometry between $M$ and $M'$ that restricted to $S$ is the identity.
\end{Def}

The theory of existence and uniqueness of primitive embeddings is vast. We recall here just the theorems that we used above.

\begin{Thm}[\cite{Nikulin79}, Thm. 1.14.4]\label{Nik_1.14.4}
Let $S$ be an even lattice of signature
$(s_+,s_-)$ and let $L$ be an even unimodular lattice of signature $(l_+, l_-)$. There exists a
unique primitive embedding of $S$ into $L$, provided the following conditions hold:
\begin{enumerate}
    \item $l_+-s_+>0$ and $l_--s_->0$;
    \item $rk(L)-rk(S)\geq 2+\ell(({A_S})_p)$ for every $p\neq 2$;
    \item if $rk(L)-rk(S)= 2+\ell(({A_S})_2)$, then $q_S=q'\oplus u_2$ or $q_S=q'\oplus v_2$, where $u_2, v_2$ are the discriminant form of the lattices $U(2)$, $D_4$ respectively. 
\end{enumerate}
\end{Thm}

\begin{Thm}[\cite{Nikulin79}, Prop. 1.15.1]\label{Nik_1.15.1}.
The primitive embeddings of a lattice $S$ with signature $(s_+, s_-)$ and discriminant form $q_S$ into an even lattice $M$ with signature $(m_+, m_-)$ and discriminant form $q_M$ are determined by the sets $(H_S, H_M, \gamma, T, \gamma_T)$, where: 
\begin{enumerate}
\item $H_S \subset A_S$ and $H_M\subset A_M$ are subgroups and $\gamma: q_S|_{H_S}\rightarrow q_M|_{H_M}$ is an isomorphism of the subgroups, preserving
the restrictions of the forms;
\item $T$, which will be the orthogonal complement to $S$ in $M$, is an even lattice with signature
$(m_+-s_+, m_- -s_-)$ and discriminant form $q_T$;
\item $\gamma_T: q_T\rightarrow -\delta$ is an isomorphism of discriminant forms, where $\delta\simeq ((q_S \oplus -q_M)|_{\Gamma^\perp})/\Gamma$, $\Gamma=H_S\oplus H_M$ (notice that $(q_S \oplus -q_M)|_\Gamma=0$ by choice of $\gamma$).
\end{enumerate}
Two such sets, $(H_S, H_M, \gamma, T, \gamma_T)$ and $(H'_S, H'_M, \gamma', T', \gamma_{T'})$ determine isomorphic
primitive embeddings if and only if $H_S = H'_S$ and there exist $\xi\in O(A_M)$ and $\psi:T\rightarrow T'$ isometries for which 
$\gamma' = \xi\circ\gamma$ and $\bar\xi\circ\gamma_T = \gamma_{T'}\circ\bar\psi$, where $\bar\xi$ is the isomorphism of discriminant forms $\delta$ and $\delta'$ induced by $id\oplus\xi$, and $\bar\psi$ is the isomorphism of discriminant forms $q_T$ and $q_{T'}$ induced by $\psi$.
\end{Thm}
 
\bibliography{references.bib}{}
\bibliographystyle{alpha}
\end{document}